\documentclass[12pt,reqno]{amsart}

\usepackage[arrow,matrix,curve]{xy}

\usepackage[dvips]{graphicx} 

\usepackage{amssymb, latexsym, amsmath, amscd, array, hyperref
}

\usepackage{breakurl}   

\newtheorem{theorem}{Theorem}[section]

\newtheorem{proposition}[theorem]{Proposition}
\newtheorem{corollary}[theorem]{Corollary}

\theoremstyle{definition}
\newtheorem{definition}[theorem]{Definition}

\numberwithin{equation}{section}

\newcommand\N {{\mathbb N}} 

\newcommand\R {{\mathbb R}}
\newcommand\Q {{\mathbb Q}}
\newcommand\Z {{\mathbb Z}} 
\newcommand\st{{\rm st}} 

\newcommand{\hr} {{}^{\mathfrak{h}}\hspace*{-0.5pt}\R}

\newcommand\astr{{{}^\ast\hspace*{-0.5pt}\R}}

\newcommand\astq{{{}^{\ast}\hspace*{-.6pt}\Q}}

\newcommand\astm{{{}^{\ast}\hspace*{-2.2pt}M}}

\newcommand\astd{{}^{\ast}\!d}

\newcommand{\halo}{{\scalebox{1}[.3]{\ensuremath{\bigcirc}}}} 

\newcommand{\halobig}{{\scalebox{2}[.3]{\ensuremath{\bigcirc}}}}

\newcommand\astb{{}^{\ast}\hspace*{-2pt}B}

\newcommand\hm{{\;\overline {\hspace*{-3pt}M}}}

\newcommand\ob{{\overline{\!B}}}

\newcommand\aob{{{}^\ast\overline{\!B}}}

\numberwithin{equation}{section}

\author{Vladimir Kanovei} \address{V. Kanovei, IPPI RAS, Moscow,
Russia}\email{kanovei@googlemail.com}

\author{Mikhail G. Katz}\address{M. Katz, Department of Mathematics,
Bar Ilan University, Ramat Gan 5290002 Israel}
\email{katzmik@macs.biu.ac.il}

\author{Tahl Nowik}\address{T. Nowik, Department of Mathematics, Bar
Ilan University, Ramat Gan 5290002 Israel}\email{tahl@math.biu.ac.il}

\title
[Metric completions, Heine--Borel, and approachability]
{Metric completions, the Heine--Borel property, and approachability}

\begin{document}

\begin{abstract}
We show that the metric universal cover of a plane with a puncture
yields an example of a nonstandard hull properly containing the metric
completion of a metric space.  As mentioned by do Carmo, a
nonextendible Riemannian manifold can be noncomplete, but in the
broader category of metric spaces it becomes extendible.  We give a
short proof of a characterisation of the Heine--Borel property of the
metric completion of a metric space~$M$ in terms of the absence of
inapproachable finite points in~$\astm$.

\end{abstract}

\maketitle

\section{Introduction}
\label{s1}

A~$p$-adic power series example of the phenomenon of inapproachability
in a nonstandard hull of a metric space~$M$ appears in Goldblatt
\cite[p.\;252]{Go98}.  Recall that a point~$x\in\astm$ is approachable
if for each~$\varepsilon\in\R^+$ there is some
(standard)~$x_\varepsilon\in M$ such
that~$\astd(x,x_\varepsilon)<\varepsilon$ (op.\;cit., p.\;236).
Otherwise~$x$ is called inapproachable.

A nonstandard hull of a metric space~$M$ can in general contain points
that need to be discarded (namely, the inapproachable ones) in order
to form the metric completion of~$M$.  We provide a more geometric
example of such a phenomenon stemming from differential geometry.  The
example is the metric universal cover of a plane with one puncture;
see Definition~\ref{d32}.

Let~$\astr$ be a hyperreal field extending~$\R$.  Denote
by~$\hr\subseteq\astr$ the subring consisting of finite hyperreals.
The ring~$\hr$ is the domain of the standard part
function~$\st\colon\hr\to\R$.  Here~$\st(x)$ for~$x\in\hr$ is the real
number corresponding to the Dedekind cut on~$\R$ defined by~$x$, via
the embedding~$\R\hookrightarrow\hr$.

Let~$\astq\subseteq \astr$ be the subfield consisting of hyperrational
numbers.  Let~$F\subseteq\astq$ be the ring of finite hyperrationals,
so that~$F=\astq\cap\hr$.  Let~$I\subseteq F$ be the ideal of
hyperrational infinitesimals.  If~$x\in\astq$ then its \emph{halo} is
the (co)set~$\text{hal}(x)=x+I\subseteq \astq$.  The following result
is well known; see e.g., \cite{Da77}.

\begin{theorem}
\label{t672} 
The ideal~$I\subseteq F$ is maximal, and the quotient
field\,~$\widehat\Q=F/I$ is naturally isomorphic to~$\R$, so that we
have a short exact sequence~$0\to I \to F \to \R \to 0$.
\end{theorem}

\begin{proof}
A typical element of~$\widehat\Q$ is a halo, namely
$\text{hal}(x)\subseteq\astq$, where each~$x\in F$ can be viewed as an
element of the larger ring~$\hr\subseteq\astr$.  Then the map
\[
\phi(\text{hal}(x))=\st(x)
\]
is the required map~$\phi\colon\widehat\Q\to\R$.  To show surjectivity
of~$\phi$, note that over~$\R$ we have
\begin{equation}
\label{e11}
(\forall\epsilon\in\R^+)(\forall{}y\in\R)(\exists{}q\in\Q)\;
[|y-q|<\epsilon].
\end{equation}
Recall that Robinson's transfer principle (see \cite{Ro66}) asserts
that every first-order formula, e.g., \eqref{e11}, has a hyperreal
counterpart obtained by starring the ranges of the bound variables.
We apply the upward transfer principle to \eqref{e11} to obtain
\begin{equation}
\label{e12}
(\forall \epsilon\in\astr^+)( \forall y\in\astr)( \exists
q\in\astq)\;\big[|y-q|<\epsilon\big].
\end{equation}
Now choose an infinitesimal~$\epsilon >0$.  Then formula~\eqref{e12}
implies that for each real number~$y$ there is a hyperrational~$q$
with~$y\approx q$, where~$\approx$ is the relation of infinite
proximity (i.e.,~$y-q$ is infinitesimal).  Therefore we
obtain~$\phi(\text{hal}(q))=y$, as required.
\end{proof}

A framework for differential geometry via infinitesimal displacements
was developed in \cite{15d}.  An application to small oscillations of
the pendulum appeared in \cite{16c}.  The reference \cite{Go98} is a
good introductory exposition of Robinson's techniques, where the
reader can find the definitions and properties of the notions
exploited in this text.  Additional material on Robinson's framework
can be found in \cite{17f}.  The historical significance of Robinson's
framework for infinitesimal analysis in relation to the work of
Fermat, Gregory, Leibniz, Euler, and Cauchy has been analyzed
respectively in \cite{18d}, \cite{19b}, \cite{18a}, \cite{17b},
\cite{17a}, and elsewhere.  The approach is not without its critics;
see e.g., \cite{13d}.

\section{Ihull construction}
\label{s62}

In Section~\ref{s1} we described a construction of~$\R$ starting
from~$\astq$.  More generally, one has the following construction.  In
the literature this construction is often referred to as the
\emph{nonstandard hull construction}, which we will refer to as the
\emph{ihull construction} (``i'' for \emph{infinitesimal}) for short.
The general construction takes place in the context of an arbitrary
metric space~$M$.

Given a metric space~$(M,d)$, we consider its natural
extension~$\astm$.  The distance function~$d$ extends to a
hyperreal-valued function~$\astd$ on~$\astm$ as usual.  The halo
of~$x\in \astm$ is defined to be the set of points in~$\astm$ at
infinitesimal distance from~$x$.

Let~$\approx$ be the relation of infinite proximity in~$\astm$.
Denote by~$F\subseteq \astm$ the set of points of~$\astm$ at finite
distance from any point of~$M$ (i.e., the galaxy of any element
in~$M$).  The quotient
\[
F/\!\approx
\]
is called the \emph{ihull} of~$M$ and denoted~$\widehat M$.  In this
terminology, Theorem~\ref{t672} asserts that the ihull of~$\Q$ is
naturally isomorphic to~$\R$.  Thus, ihulls provide a natural way of
obtaining completions; see Morgan (\cite{Mo16}, 2016) for a general
framework for completions.  We will exploit the following notation for
halos.
\begin{definition}
We let~$\overset{\halo}{x}$ be the halo of~$x\in \astm$.
\end{definition}

In general the ihull~$\widehat M$ of a metric space~$M$ consists of halos
$\overset{\halo}{x}$, where~$x\in F$, with distance~$d$ on~$\widehat M$
defined to be
\begin{equation}
\label{e681}
d \left( \overset{\halo}{x},\overset{\halo}{y}\right)=\st(\astd(x,y)).
\end{equation}

Note that that~$M$ may be viewed as a subset of~$\widehat{M}$.  Hence
it is meaningful to speak of the closure of~$M$ in~$\widehat{M}$.  We
will denote such closure~$cl(M)$ to distinguish it from the abstract
notion of the metric completion~$\hm$ of~$M$ (see Section~\ref{s5}).
The closure~$cl(M)$ is indeed the completion in the sense that it is
complete and~$M$ is dense in it.  The metric completion~$\hm$ of~$M$
is the approachable part of~$\widehat M$, and coincides with the
closure~$cl(M)$ of~$M$ in~$\widehat M$; see \cite[Chapter 18]{Go98}
for a detailed discussion.

\section{Universal cover of plane with a puncture}

The ihull~$\widehat M$ may in general be larger than the metric
completion~$\hm$ of~$M$.  An example of such a phenomenon was given in
\cite[p.\;252]{Go98} in terms of~$p$-adic series.  We provide a more
geometric example of such a phenomenon stemming from differential
geometry.

We start with the standard flat metric~$dx^2+dy^2$ in
the~$(x,y)$-plane, which can be written in polar
coordinates~$(r,\theta)$ as~$dr^2+r^2d\theta^2$ where~$\theta$ is the
usual polar angle in~$\R/2\pi\Z$.

\begin{definition}
\label{d32}
Let~$M$ be the metric universal cover of~$\R^2\setminus\{0\}$ (the
plane minus the origin), coordinatized by~$(r,\zeta)$ where~$r>0$
and~$\zeta$ is an arbitrary real number.
\end{definition}

In formulas,~$M$ can be given by the coordinate
chart~$r>0$,~$\zeta\in\R$, equipped with the metric
\begin{equation}
\label{e691}
dr^2+r^2d\zeta^2.
\end{equation}
Formula~\eqref{e691} provides a description of the metric universal
cover of the flat metric on~$\R^2\setminus\{0\}$, for which the
covering map~$M\to\R^2\setminus\{0\}$
sending~$(r,\zeta)\mapsto(r,\theta)$ induces a local Riemannian
isometry, where~$\theta$ corresponds to the coset~$\zeta+2\pi\Z$.
Recall that a number is called \emph{appreciable} when it is finite
but not infinitesimal.

\begin{theorem}
Points of the form~$\overset{{}\halobig}{(r,{\zeta})}$ for
appreciable~$r$ and infinite~$\zeta$ are in the ihull\,~$\widehat M$
but are not approachable from~$M$.
\end{theorem}

\begin{proof}
The distance function~$d$ of~$M$ extends to the ihull~$(\widehat M,d)$
as in~\eqref{e681}.  Here points of~$\widehat M$ are halos in the
finite part of~$\astm$.  Notice that in~$\widehat M$ the origin has
been ``restored" and can be represented in coordinates~$(r,\zeta)$ by
a point~$(\epsilon,0)$ in~$\astm$ where~$\epsilon>0$ is infinitesimal.

Consider a point~$(1,\zeta)\in\astm$ where~$\zeta$ is infinite.  Let
us show that the point~$(1,\zeta)$ is at a finite
distance~${}^\ast\!d$ from the point~$(\epsilon,0)$; namely the
standard part of the distance is~$1$.  Indeed, the triangle inequality
applied to the sequence of points
$(1,\zeta)$,~$(\tfrac{1}{\zeta^2},\zeta)$,
$(\tfrac{1}{\zeta^2},0)$,~$(\epsilon,0)$ yields the bound
\[
{}^\ast\!d\big((1,\zeta),(\epsilon,0)\big) \leq (1-\tfrac{1}{\zeta^2})+
\tfrac{1}{\zeta^2}\zeta + |\tfrac{1}{\zeta^2}-\epsilon|\approx 1.
\]
Therefore
\[
d\Big(\overset{{}\halobig}{(1,\zeta)},
\overset{{}\halobig}{(\epsilon,0)}\Big)\leq1
\]
by \eqref{e681}.  Hence~$\overset{{}\halobig}{(1,\zeta)}\in
\widehat M$.

On the other hand, let us show that the point
$\overset{{}\halobig}{(1,\zeta)}\in\widehat M$ is not approachable
from~$M$.  Consider the rectangle~$K$ defined by the image in
$\widehat M$ of
\[
{}^\ast\hspace{-1pt}[\tfrac{1}{2},2]
\times{}^\ast\hspace{-1pt}[\zeta-1,\zeta+1]\subseteq \astm.
\]
The metric~$d$ of~$\widehat M$ restricted to the rectangle~$K$
dominates the product metric~$dr^2+\frac{1}{4}d\zeta^2$
by~\eqref{e691}.  The boundary of~$K$ separates the interior of~$K$
from the complement of~$K$.  Thus, to reach the finite part one must
first traverse the boundary.  Therefore $K$ includes the metric ball
of radius~$\frac{1}{2}$ centered at~$\overset{{}\halobig}{(1,\zeta)}$.
This ball contains no standard points.  Hence the
point~$\overset{{}\halobig}{(1,\zeta)}\in\widehat M$ is not
approachable.
\end{proof}

Note that what is responsible for the inapprochability is the fact
that the closure~$\hm \subseteq \widehat M$ does not have the
Heine--Borel property (and is not even locally compact).  Namely, the
boundary of the metric unit ball in~$\hm$ centered at the
origin~$\overset{{}\halobig}{(\epsilon,0)}$ is a line.

Do Carmo (\cite{Do92}, 1992, p.\;152) views the universal cover~$M$ of
the plane with a puncture as an example of a Riemannian manifold that
is \emph{nonextendible} but not complete.  Indeed,~$M$ is
nonextendible in the category of Riemannian manifolds, but~$M$ is
extendible in the category of metric spaces, in such a way that near
the ``extended'' origin~$\overset{{}\halobig}{(\epsilon,0)}\in
\hm\subseteq \widehat M$, the Heine--Borel property is violated.  In
Section~\ref{s5} we show that such a result holds more generally.

\section
{An approachable criterion for the Heine--Borel property}
\label{s5}

For the sake of completeness we provide a short proof of a relation
between approachability and the Heine--Borel property for metric
spaces.  For related results in the context of uniform spaces see
Henson--Moore \cite{He73}, \cite{He74} (but note that they use a
different notion of ``finiteness'' for a point $x\in \astm$).  For a
study of the relation between the Heine--Borel property and local
compactness, see \cite{Wi87}.

We show that the Heine--Borel property for the completion of a metric
space has a characterisation in terms of the absence of finite
inapproachable points; see Theorem~\ref{b}.  The following result
appears in Luxemburg \cite[Theorem 3.14.1, p.\;78]{Lu69} and
Hurd--Loeb \cite[Proposition 3.14]{Hu85}; cf.\;Davis
\cite[Theorems\;5.19 and 5.20, p.\;93]{Da77}.

\begin{proposition}
\label{a}
Let~$M$ be a metric space.  Then the following two properties are
equivalent:
\begin{enumerate}
\item
every approachable point in\,~$\astm$ is nearstandard;
\item
$M$ is complete.
\end{enumerate}
\end{proposition}

\begin{definition}
A metric space~$M$ is \emph{Heine--Borel} (HB) if every closed and
bounded subset of~$M$ is compact.
\end{definition}
We fix a point~$p \in M$.  Let~$n\in\N$.  Let~$B_n=\{x\in M\colon
d(x,p)\leq n \}$.  Clearly~$M$ is HB if and only if the sets~$B_n$ are
all compact.  Let~$\hm$ be the completion of~$M$.  Let~$\ob_n=\{x\in
\hm \colon d(x,p)\leq n \}$, for the same fixed~$p\in M$.  By
transfer, we have~$\astb_n = \{x\in \astm \colon d(x,p)\leq n \}$, and
similarly~$\aob_n = \{x\in {}^\ast\!\hm \colon d(x,p)\leq n \}$.

Clearly, an HB metric space is complete (given a Cauchy sequence, find
a convergent subsequence in the closure of its set of points).

\begin{theorem}
\label{b}
Let~$M$ be a metric space.  The following three properties are
equivalent:
\begin{enumerate}
\item
Every finite point in~$\astm$ is approachable;
\item
the completion~$\hm$ is Heine--Borel;
\item
$\hm=\widehat M$ (the metric completion is already all of the ihull).
\end{enumerate}
\end{theorem}

\begin{proof}
Assume~$\hm$ is HB.  Let~$a\in \astm$ be finite.  Then we
have~$a\in\astb_n\subseteq\aob_n$ for some~$n\in \N$.  Since~$\hm$ is
assumed to be HB, the ball~$\ob_n$ is compact.  Hence there is a
point~$x\in\ob_n$ with~$x\approx a$. Now let~$\epsilon>0$ be standard.
Since~$M$ is dense in~$\hm$ there is a point~$y\in M$ such
that~$d(y,x)<\epsilon$, and therefore~$d(y,a)<\epsilon$.

Conversely, assume every finite point in~$\astm$ is approachable. As a
first step we show that every finite point in~$^*\!\hm$ is
approachable.  Let~$a\in \aob_n$ and fix a standard~$\epsilon>0$.
Since~$M$ is dense in~$\hm$, the following holds for our fixed~$n$
and~$\epsilon$:
\[
(\forall x\in \ob_n) (\exists y\in B_{n+1})[d(x,y)<\epsilon].
\]
By transfer we obtain
\begin{equation}
\label{e41}
(\forall x\in\aob_n) (\exists y\in \astb_{n+1})[d(x,y)<\epsilon].
\end{equation}
Applying \eqref{e41} with~$x=a$, we obtain a point~$b\in\astb_{n+1}$
with~$d(a,b)<\epsilon$.  Every finite point in~$\astm$ is approachable
by assumption.  Therefore there is a point~$x \in M\subseteq\hm$
with~$d(x,b)<\epsilon$.  Thus~$d(x,a)<2\epsilon$, showing that every
finite point in~$^*\!\hm$ is approachable.

We now prove that~$\hm$ is HB by showing that each~$\ob_n$ is compact.
Let~$a\in\aob_n$.  We need to find a point~$x\in\ob_n$ with~$x\approx
a$.  We have shown above that~$a$ is approachable.  By Proposition
\ref{a},~$a$ is nearstandard, i.e., there is a point~$x\in \hm$
with~$x\approx a$.  Since~$d(a,p)\leq n$, and~$x,p$ are both standard,
we also have~$d(x,p)\leq n$, i.e.,~$x\in\ob_n$.
\end{proof}

Combining Proposition \ref{a} and Theorem \ref{b}, we obtain the
following corollary, which also appears in Goldbring
\cite[Proposition\;9.23]{Go14}.

\begin{corollary}
Let~$M$ be a metric space.  The following two properties are
equivalent:
\begin{enumerate}
\item
Every finite point in~$\astm$ is nearstandard;
\item
$M$ is Heine--Borel.
\end{enumerate}
\end{corollary}

\section*{Acknowledgments}

We are grateful to the referees for helpful comments on an earlier
version of the manuscript.  V.\;Kanovei was partially supported by
RFBR grant no 18-29-13037.

\bibliographystyle{amsalpha}

\end{document}